\documentclass[dvipsnames]{amsart}[12]

\usepackage[a4paper]{geometry}
\usepackage[backend=biber, datamodel=mrnumber, maxbibnames=99, sortcites]{biblatex}
\usepackage{amsmath}
\usepackage{amsthm}
\usepackage{mathrsfs}
\usepackage{amsfonts}
\usepackage{amssymb}
\usepackage{amscd}
\usepackage{array}
\usepackage{amssymb}
\usepackage[all, cmtip]{xy}
\usepackage{tikz}
\usetikzlibrary{arrows}
\usetikzlibrary{cd}
\usepackage{ulem}
\usepackage{enumitem}

\usepackage{thm-restate}
\usepackage{hyperref}

\usepackage{xcolor}
\usepackage{quiver}

\newtheorem{theorem}{Theorem}[subsection]
\newtheorem{lemma}[theorem]{Lemma}
\newtheorem{fact}[theorem]{Fact}

\newtheorem{definition}[theorem]{Definition}
\newtheorem{corollary}[theorem]{Corollary}
\newtheorem{proposition}[theorem]{Proposition}
\newtheorem{remark}[theorem]{Remark}

\newcommand{\PP}{\mathbb{P}}

\newcommand{\EE}{\mathbb{E}}
\newcommand{\LL}{\mathbb{L}}
\renewcommand{\AA}{\mathbb{A}}

\newcommand{\OO}{\mathscr{O}}

\newcommand{\Hom}{Hom}

\newcommand{\FF}{\mathbb{F}}

\newcommand{\fX}{\mathfrak X}

\newcommand{\relSpec}{\underline{\mathrm{Spec}}}
\newcommand{\VV}{\mathbb{V}}
\newcommand{\gray}[1]{{\color{gray}#1}}

\renewcommand{\AA}{\mathbb{A}}

\newcommand{\fM}{\mathfrak M}

\newcommand{\vir}{\mathrm{vir}}

\newcommand{\id}{\mathrm{id}}

\newcommand{\Rachel}[1]{{\color{NavyBlue}#1}}

\newcommand{\dsp}[2]{M_{#1}^{\circ}#2}

\newcommand{\rank}{\mathrm{rank}}
\newtheorem{example}{Example}

\DeclareMathOperator{\Spec}{Spec}

\renewcommand{\Hom}{\mathrm{Hom}}

\newcommand{\mf}{\mathfrak}

\newif\ifmoditem
\newcommand{\setupmodenumerate}{%
  \global\moditemfalse
  \let\origmakelabel\makelabel
  \def\moditem##1{\global\moditemtrue\def\mesymbol{##1}\item}%
  \def\makelabel##1{%
    \origmakelabel{##1\ifmoditem\rlap{\mesymbol}\fi\enspace}%
    \global\moditemfalse}%
}

\newcommand{\mls}{\mathscr}
\newcommand{\mc}{\mathcal}

\newcommand{\Dqc}{\mathsf{D}_{\mathrm{qc}}}
\newcommand{\fp}{\mathfrak{p}}

\begin{document}

\subjclass[2020]{14C15, 14A20, 14N35}

\title{Siebert's formula for virtual pullbacks}

\author{Xuanchun Lu}
\author{Rachel Webb}

\begin{abstract} For a DM-type morphism of algebraic stacks $f: X \to  Y$ with a relative perfect obstruction theory $\EE \to \LL_{X/Y}$ having a global resolution, we give a formula for the induced virtual pullback in terms of $\EE$ and $f$.
We apply our formula to prove functoriality properties of the virtual fundamental class in Gromov-Witten theory.
\end{abstract}

\maketitle

\section{Introduction}

Let $f: X \to Y$ be a DM-type morphism of algebraic stacks with $X$ stratified by global quotient stacks, and let $\phi: \EE \to \LL_{X/Y}$ be a perfect obstruction theory in the sense of Behrend-Fantechi \cite{BF}.
Recall that given a projective morphism $T \to Y$,
    a vector bundle $E \to T$,
    and an integral closed substack $V \subseteq E$,
    we obtain a cycle $[(V, E, T)_{Y}]$
    in Kresch's Chow group $A_{\ast}(Y)$.
Our main theorem is the following formula for the virtual pullback $f^!_\phi$.
\begin{theorem}[Theorem \ref{thm:siebert2}]\label{thm:siebert-intro} Assume that $\EE$ is globally isomorphic to a 2-term complex of locally free sheaves on $X$. Then virtual pullback is computed by the formula
\[
f_{\phi}^{!}(\;[(V, E, T)_Y]\;) =
        \left\{ c(\EE^{\vee})^{-1} \cap s(f; V, E, T)\right\}_{\dim(V, E, T) + \rank(\EE)},
\]
where $c(-)$ is the total Chern class, $s(f; V, E, T)$ is an element of a completion of $A_*(X)$ determined by $f$ and the triple $(V, E, T)$ (Definition \ref{def:segre}), and $\{-\}_d$ means to take the dimension-$d$ piece.
\end{theorem}
Our theorem is motivated by Siebert's formula \cite[Thm 4.6]{siebert} for the virtual fundamental class associated to $\phi$, a formula that is in turn motivated by Fulton's formula \cite[Example 4.1.8]{fulton} for the intersection of a cone with the zero section of a vector bundle. Our theorem generalizes Siebert's formula in three ways.
\begin{itemize}
\item We work with algebraic stacks and Kresch's theory \cite{kresch} of cycle class groups.
\item We drop the requirement that $X$ be  quasi-projective (in particular, globally embeddable into a stack smooth over $Y$). Along the way, we interpret Fulton's canonical class of a morphism $f: X \to Y$ \cite[Example 4.2.6]{fulton} as the Segre class of the intrinsic normal cone of $f$, and generalize the definition to all situations when that cone has a global presentation.
\item We extend the formula to virtual pullbacks. A key ingredient is a cycle-level formula (Proposition \ref{prop:sigma-on-cycles}) for specialization to the normal cone for DM-type morphisms of algebraic stacks.
\end{itemize}

\noindent
We expect Theorem \ref{thm:siebert-intro} will have many applications due to the following corollary.
\begin{corollary}
Assume that $\EE$ is globally isomorphic to a 2-term complex of locally free sheaves on $X$. The virtual pullback $f^!_\phi$ depends only on $\EE$ and $f$, not on the morphism $\phi$.
\end{corollary}
As an example, we use Theorem \ref{thm:siebert-intro} to reprove the Splitting Axiom in Gromov-Witten theory (\cite[Def 7.1(3)]{BM96} and \cite[Axiom III]{BehrendGW}). The key technical lemma (Lemma \ref{lem:key-lemma}) in our argument may be read as a statement  of ``functoriality in the source curve'' of the Gromov-Witten obstruction theory. Indeed, we prove Lemma \ref{lem:key-lemma} in greater generality than is needed for the Splitting Axiom as we anticipate applications to other settings, including wall-crossing theorems for contractions of source curves and Gromov-Witten theory of tame Artin stacks.

\subsection{Conventions}

Unless otherwise noted, all stacks $X$ are of finite type over a ground field and algebraic in the sense of Laumon and Moret-Bailly \cite{LMB}, as this is the hypothesis of Kresch \cite{kresch}; the associated cycle class group is denoted $A_*(X)$.
A vector bundle on $X$ is $\relSpec_X(Sym(\mls E))$ where $\mls E$ is a finite rank locally free sheaf on $X$.

\subsection{AI disclosure}
The authors used GPT-5.6 Sol to proofread this paper for mathematical and typographical errors.

\subsection{Acknowledgements}
The second author thanks Dhruv Ranganathan for pointing out the utility of Siebert's formula and Felix Janda for helpful discussions on the proof of the Splitting Axiom. The authors are grateful to Cristina Manolache for helpful conversation.
The first author was supported by the ERC Advanced Grant MSAG, and the second author was partially supported by the NSF grant DMS 2501528.
This collaboration began at the Bootcamp of the 2025 Summer Research Institute in Algebraic Geometry.

\section{Specialization to the normal cone}\label{sec:sigma}

The goal of this section is to  understand the specialization to the normal cone used in \cite[Sec 5.1]{kresch}.

\subsection{Kresch's  Chow groups}
\label{sec: kresch chow}
    Given an algebraic stack $X$ of finite type over a ground field and algebraic in the sense of \cite{LMB},
    recall that a cycle on $X$ in the sense of Kresch \cite{kresch}
    is represented by an integral linear combination of triples $(V, E, T)_X$,
    where
    \begin{itemize}
        \item $g_{T} \colon T \to X$ is a projective morphism,
        \item $p_{E} \colon E \to T$ is a vector bundle, and
        \item $i_V: V \hookrightarrow E$ is an integral closed substack.
    \end{itemize}
    The \emph{dimension} of the cycle $(V, E, T)_X$ is $\dim (V) - \rank(E)$; that is, $[(V, E, T)_X]$ lives in $A_{\dim (V) - \rank(E)}(X)$.
If  $\alpha$ is an integral linear combination of closed substacks of $E$, we write $(\alpha, E, T)_X$ for the corresponding integral linear combination of triples.
The cycle class groups $A_*(X)$
    have flat pullback and projective pushforwards
    that enjoy the usual compatibility properties.

    \begin{fact}\label{fact:cycle-is}
    In $A_*(X)$ we have an equality
    \[[(V, E, T)_X] = g_{T\ast}(p_{E}^{\ast})^{-1}i_{V*}([(V, V, V)_V]).\]
    \end{fact}

    \begin{fact}\label{fact:zero-fiber} If $X$ has a morphism $X \to \AA^1$, we denote by $X_0$ the fiber over zero. There is a morphism $A_*(X) \to A_*(X_0)$, namely intersection with the zero fiber, that sends $[(V, E, T)_X]$ to $[([V_0], E_0, T_0)_{X_0}]$ if $V$ is not contained in $E_0$, and it sends $[(V, E, T)_X]$ to zero otherwise (\cite[p. 502]{kresch} and \cite[Remark 2.3]{fulton}).
\end{fact}

We now recall specialization to the normal cone. Let $X \to Y$ be a DM-type morphism of algebraic stacks in the sense of \cite[Def 2.1]{manolache-pullback}. There  is  a flat deformation space $M^\circ_XY \to \PP^1$ whose  fiber over 0 is the intrinsic normal cone $ C_XY$, and the complement of this zero fiber is isomorphic to $Y \times \AA^1$ (\cite[Proof of Thm 2.31]{manolache-pullback}). The space $M^\circ_X Y$ is an algebraic stack in the sense of \cite[Tag 026O]{stacks-project}, but the diagonal may not be separated.
Specialization to the normal cone is the homomorphism $\sigma: A_*(Y) \to A_*(C_X Y)$ defined as the ``composition''
\[
A_*(Y) \to A_*(Y \times \AA^1) \leftarrow A_*(M^\circ_X Y) \to A_*(C_XY)
\]
where
\begin{itemize}
\item The morphism $A_*(Y) \to A_*(Y \times \AA^1)$ is flat pullback.
\item The morphism $A_*(M^\circ_X Y) \to  A_*(Y \times \AA^1)$ is flat pullback; in the above ``composition,'' we send a class $\alpha \in A_*(Y \times \AA^1)$ to any lift along this morphism.
The difference between any two such lifts by excision is pushed forward from $ C_X Y$ and by Fact \ref{fact:zero-fiber} maps to zero under the next morphism.
\item The morphism $A_*(M^\circ_X Y) \to A_*( C_XY)$ is intersection with the zero fiber.
\end{itemize}

\begin{remark}
This definition of $\sigma$ assumes  that at least some portion of Kresch's Chow theory goes through for the stack $M^\circ_X Y$.
However, as noted in \cite[footnote 1 on p.529]{kresch}, the diagonal of this stack is not separated when $X \to Y$ is not representable, and consequently $M^\circ_X Y$ is outside the framework considered in \cite{kresch}. Nevertheless this definition of $\sigma$ has been used in both \cite{KKP} and \cite{manolache-pullback}. As far as we can tell, the only place where the non-separated diagonal of $M^\circ_X Y$ could cause issues is in the use of excision: \cite[Prop 2.3.6]{kresch} ultimately relies on the ability to extend coherent sheaves from open substacks of $M^\circ_X Y$. Fortunately the proof of \cite[Prop 2.3.6]{kresch} goes through verbatim if one assumes the inclusion of the open substack is quasi-compact and quasi-separated, so that pushforward preserves quasicoherence.
In the definition of $\sigma$, we apply excision to the open substack $Y \times \AA^1 \to M^\circ_X Y$ that is the complement of an effective Cartier divisor, hence the inclusion is quasi-compact and quasi-separated.
\end{remark}

The main result of Section \ref{sec:sigma} is to prove the following explicit description of $\sigma$.

\begin{restatable}{proposition}{sigmacyc}
    \label{prop:sigma-on-cycles}
    Let $X \to Y$ be a DM-type morphism of algebraic stacks.
    Let $[(V, E, T)_Y]$ be a cycle on $Y$,
    with $V \subset E$ an integral substack.
    Let $X_T = X \times_{Y} T$, $X_E = X \times_{Y} E$,
    and $X_V = X \times_{Y} V$.
    Then
    \[\sigma(\; [(V, E, T)_Y]\;) =
        [(\,[{C}_{X_V}V], \;{C}_{X_E}E, \; {C}_{X_T}T\,)_{C_X Y}]. \quad \quad \quad
    \]
\end{restatable}

\noindent
We will prove the proposition after some preliminary results on deformation spaces.

\subsection{Properties of deformation spaces}\label{sec:prop-of-mxy}
For a DM-type morphism $X \to Y$ of algebraic stacks, the deformation space $M^\circ_X Y$ has a morphism to $Y$, and a commuting square with $X' \to Y'$ mapping to $X \to Y$ induces a morphism $M^\circ_{X'} Y' \to M^\circ_X Y$.
The definition of the deformation space is summarized below.

\begin{definition} \label{def:mxy} The definition of $M^\circ _X Y$ is iterative, depending on properties of $X \to Y$.
\begin{itemize}
    \item[(i)] Closed embedding \cite[p.~87]{fulton}. We will use both the definition in terms of blowups (\cite[p.~87]{fulton}) and the one in terms of graph closure (\cite[Remark~5.1.1]{fulton}) below.
\item[(ii)] Locally closed embedding  \cite[p.~489]{kresch2}. Factor $X \to Y$ as $X \to U \to Y$ where $X \to U$ (resp. $U \to Y$) is a closed (resp. open) embedding. Then $M^\circ_X Y$ is obtained by gluing $M^\circ_X U$ and $Y \times \AA^1$ along $U \times \AA^1$.
\item[(iii)] Local immersion  \cite[p. 489]{kresch2}. One forms a commuting diagram
\begin{equation}\label{eq:uvxy}
\begin{tikzcd}
U \arrow[r]  \arrow[d] & V \arrow[d] \\
X \arrow[r] & Y
\end{tikzcd}
\end{equation}
with $U \to  V$ a closed embedding and vertical maps \'etale and surjective.
Setting $R = U \times_X U$ and $S = V \times_Y V$, one can show that $R \to S$ is a locally closed embedding and that there is an \'etale groupoid $[M^\circ_R S \rightrightarrows M^\circ_U V]$. The quotient (algebraic space) of this groupoid is defined to be $M^\circ_X Y$.
\item[(iv)] DM-type \cite[Thm 2.31]{manolache-pullback} (cf. \cite[Sec 5]{kresch}). Similar to the previous case, one forms a commuting diagram \eqref{eq:uvxy}  where $U$ and $V$ are schemes, the vertical maps are smooth and surjective, and $U \to V$ is a closed embedding. Define $R$ and $S$ as before. This time $R \to S$ is a local immersion of schemes and $M^\circ_X Y$ is the quotient of the smooth groupoid $[M^\circ_R S \rightrightarrows M^\circ_U V].$
\end{itemize}
\end{definition}

\begin{lemma}\label{lem:flat-basechange}
Let $X \to Y$ be a DM-type morphism of algebraic stacks, let $Y' \to Y$ be a flat morphism representable by schemes, and set $X' := X \times_Y Y'$. Then there is a fibered square
 \[
        \begin{tikzcd}
            M_{X^{\prime}}^{\circ}Y^{\prime} \arrow[d] \arrow[r] & M_{X}^{\circ}Y \arrow[d] \\
            Y^{\prime} \arrow[r] & Y.
        \end{tikzcd}
    \]
\end{lemma}
\begin{proof}
We prove the lemma inductively in each of the four cases listed in Definition \ref{def:mxy}.

When $X \to Y$ is a closed embedding, recall that $M^\circ_X Y$ is the complement in $Bl_{X \times \{0\}} Y \times \PP^1$ of an embedded copy of $Bl_XY$. Since blowups commute with flat pullback, a straightforward computation shows that the lemma holds in this case.

When $X \to Y$ is locally closed, write $X \to  U \to Y$ where $X \to U$ is a closed embedding and $U \to Y$ is an open embedding. By the previous case, the open sets $M^\circ_X U$ and $Y \times \AA^1$ that cover $M^\circ_X Y$ are compatible with flat pullback, as is their intersection $U \times \AA^1.$

When $X \to Y$ is a local immersion, the proof is morally the same as the previous case, but we explain in some detail. Let $U$, $V$, $R$, and $S$ be as in \eqref{eq:uvxy}, and set $V' = V \times_Y Y'$, $U' = U \times_Y Y'$, $S' = V' \times_{Y'} V'$, and $R' = U' \times_{X'} U'$. We have a tower of commuting cubes
\begin{equation}\label{eq:tower}\begin{tikzcd}[row sep={15,between origins}, column sep={30,between origins}]
      & R' \ar{rr}\arrow[dd, shift right] \arrow[dd, shift left]\ar{dl} & &R\arrow[dd, shift right] \arrow[dd, shift left]\ar{dl} \\
    S' \ar[crossing over]{rr} \arrow[dd, shift right] \arrow[dd, shift left] & & S \\
      & U'  \ar{rr} \ar{dl} \arrow[dd]& &  U \ar{dl}\arrow[dd] \\
    V' \ar{dd} \ar[crossing over]{rr} && V \ar[from=uu,crossing over, shift left] \ar[from=uu,crossing over, shift right] \\
    &X' \ar{rr} \ar{dl} && X \ar{dl}  \\
Y' \ar{rr} && Y \ar[from=uu,crossing over]
\end{tikzcd}\end{equation}
where all front, back, and horizontal squares are fibered, but the left and right faces may not be. Since $U \to V$ and $R \to S$ are locally closed embeddings, it follows from the previous case that $M^\circ_{U'} V'$ (resp. $M^\circ_{R'} S'$) is the base change of $M^\circ_U  V$ (resp. $M^\circ_{R} S$) along $Y' \to Y$. Since formation of the quotient of an \'etale groupoid commutes with base change,
the quotient $M^\circ_{X'} Y'$ of $[M^\circ_{R'} S' \rightrightarrows M^\circ_{U'} V']$ is the base change of $M^\circ_X Y$ along $Y' \to Y$, as desired.

When $X \to Y$ is DM-type, the inductive argument runs as in the previous case.
\end{proof}

The next lemma says that compatibility of deformation spaces and closed embeddings is analogous to that of blowups and closed embeddings (cf. \cite[{}22.2.7]{vakil}).
\begin{lemma}\label{lem:closed-basechange}
Let $f: X \to Y$ be a DM-type morphism of algebraic stacks, let $Y'  \to Y$ be a closed embedding, and set $X' = X \times_Y Y'$. Then $M^\circ_{X'} Y'$ is the scheme theoretic image of the composition
\begin{equation}
\label{eq:i}
Y' \times \AA^1 \to Y \times \AA^1 \to M^\circ_X Y.
\end{equation}
\end{lemma}
\begin{proof}
Let $i$ denote the composition \eqref{eq:i}. Then $i$ is quasi-compact, being the composition of a closed embedding and the  base change of the quasi-compact open immersion $\AA^1 \to \PP^1$. It follows from \cite[Tag 0CMK]{stacks-project} that
\begin{equation}\label{eq:fact}
\text{formation of the scheme-theoretic image of $i$ commutes with flat base change.}
\end{equation}
From here, we prove the lemma in each of the four cases in Definition \ref{def:mxy}.

When $X \to Y$ is a closed embedding, by \eqref{eq:fact} we can work Zariski locally on $Y$. So assume $Y$ is affine
    and $X$ is the zero locus of a section $s$ of a (trivial) vector bundle $E$.
    The graph construction of $M^{\circ}_X Y$ says that
    $M^\circ_{X} Y$ is the closure of the image of the morphism
    \[
        Y \times (\AA^1 \setminus \{0\}) \to E \times \AA^1 \quad \quad \quad  \quad (y, t) \mapsto (t^{-1}s(y), t).
    \]
    We have a closed embedding $Y' \to Y$ and
    $s$ restricts to the section $s'$ of $E|_{Y'}$ defining $X'$.
    Therefore the canonical morphism $M^{\circ}_{X'}Y' \to M^\circ_X Y$
    fits into a commuting diagram
    \[
        \begin{tikzcd}
            M^\circ_{X'}Y' \arrow[r]  \arrow[d, hookrightarrow] & M^{\circ}_X Y \arrow[d, hookrightarrow] \\
            E|_{Y'} \times \AA^1 \arrow[r, hookrightarrow] &  E \times \AA^1
        \end{tikzcd}
    \]
    where all hooked arrows are known  to be closed embeddings.
    It follows from \cite[Tag 07RK(3)]{stacks-project} that
    the top horizontal arrow is also a closed embedding.

    The remainder of the proof moves inductively through the last three cases in Definition \ref{def:mxy} as in the proof of Lemma \ref{lem:flat-basechange}. We write out the argument when $X \to Y$ is a local immersion. Define $U'$, $V'$, $R'$, and $S'$ as in \eqref{eq:tower}. Consider the commuting diagram
  \[
        \begin{tikzcd}[row sep=15]
            S^{\prime} \times \AA^{1} \arrow[r]\arrow[d] &
                M_{R^{\prime}}^{\circ}S^{\prime} \arrow[r]\arrow["{s^{\prime}}",d]&
                M_{R}^{\circ}S \arrow["s", d] \\
            V^{\prime} \times \AA^{1} \arrow[d] \arrow[r] &
                M_{U^{\prime}}^{\circ}V^{\prime} \arrow[r]\arrow[d]&
                M_{U}^{\circ}V\arrow[d]\\
                X' \times \AA^1 \arrow[r] & M^\circ_{X'} Y' \arrow[r] & M^\circ_X Y
        \end{tikzcd}
    \]

    The three rectangles with no corners in the middle column are fibered, and the top two rows represent scheme-theoretic images by the base of locally closed embeddings. By \eqref{eq:fact} we have that the top right square is fibered. Hence the bottom right square is fibered by \cite[Tag 04ZN]{stacks-project}, and the bottom row represents a scheme-theoretic image by \eqref{eq:fact} again.
\end{proof}

\subsection{Proof of Proposition~\ref{prop:sigma-on-cycles}}

We can now prove the desired formula for $\sigma$.

\sigmacyc*

\begin{proof}
We first show that $(\dsp{X_{V}}{V}, \dsp{X_{E}}E, \dsp{X_{T}}T)$
            defines a Kresch cycle $\alpha$ on $\dsp{X}{Y}$, or in other words
 \begin{itemize}
        \item[(i)] The canonical map $M^\circ_{X_T} T \to M^{\circ}_X Y$ is
            projective.
        \item[(ii)] The canonical map $M^\circ_{X_E} E \to M^\circ_{X_T} T$ is
            a vector bundle.
        \item[(iii)] The canonical map $M^\circ_{X_V} V \to M^\circ_{X_E} E$ is
            a closed embedding.
            \item[(iv)] The stack $M^\circ_{X_V} V$ is integral.
    \end{itemize}
    Item (ii) is a consequence of
    Lemma~\ref{lem:flat-basechange},
    whereas item (iii) is a consequence of
    Lemma~\ref{lem:closed-basechange}.
    Item (i) follows by combining (ii) and  (iii),
    since a projective morphism factors as a closed embedding
    followed by a flat morphism. For (iv), note that by Lemma \ref{lem:closed-basechange} we have that $M^\circ_{X_V} V$ is the closure of the composition $i: V \times \AA^1 \to E \times \AA^1 \to M^\circ_{X_E} E$. Since $V \times \AA^1$ is integral and $i$ is quasi-compact (see \eqref{eq:fact}), by \cite[Tag 0CML]{stacks-project} the closure $M^\circ_{X_V} V$ is also integral.
We now apply $\sigma$ to $[(V, E, T)_Y]$.
\begin{itemize}
\item The flat pullback of $[(V, E, T)_Y]$ to $Y \times \AA^1$ is by definition $\beta := [(V \times \AA^1, E \times \AA^1, T \times \AA^1)_{Y \times \AA^1}]$.
\item The flat pullback of $\alpha$ to $Y \times \AA^1$ is also $\beta$.
\item The intersection of $\alpha$ with the zero fiber is $[( [C_{X_V} V], C_{X_E}E, C_{X_T} T)_{C_XY}]$. This follows from Fact \ref{fact:zero-fiber} and \cite[Thm 2.31]{manolache-pullback}.
\end{itemize}

\end{proof}
\section{Segre classes of cone stacks and a formula for virtual pullback}

\subsection{Segre classes of cones}

Let $X$ be an algebraic stack, let $C$ be a cone on $X$ as in \cite{BF}, and let $\PP(C \oplus 1)$ be its projective closure. Mimicking \cite[Definition 17.2.1]{AF}
in the case when $X$ is a global quotient stack, we define the total Segre class of $C$ by
\[
s(C) := q_{\mathbb{P}(C \oplus 1)*}\left( \sum_{i \geq 0} c_1(\mls  O_{\mathbb{P}(C \oplus 1)}(1))^i \cap [\mathbb{P}(C \oplus 1)] \right) \quad \quad \quad \in \prod  A_*(X)
\]
where $q_{\mathbb{P}(C \oplus 1)}: \mathbb{P}(C \oplus 1) \to X$ is the structure morphism and  $[\PP(C \oplus 1)]$ is the fundamental class $[(\PP(C \oplus 1), \PP(C \oplus 1), \PP(C \oplus 1))_{\PP(C \oplus 1)}]$ in Kresch's group $A_*(\PP(C \oplus 1))$. The group $\prod A_*(X)$ where $s(C)$ lives is the direct product of the groups $A_k(X)$, since on an algebraic stack $s(C)$ could be nonzero in infinitely many degrees.

\subsection{Segre classes of cone stacks}\label{sec:segre1}

Let $X$ be an algebraic stack. Cone stacks over $X$ were defined in \cite[Def 1.8]{BF}.
If $\mf C \to X$ is a cone stack, a \emph{global presentation} of $\mf C$ is an isomorphism $\mf C \simeq [C/T]$ where $C \to X$ is a cone, $T \to X$ is a vector bundle, and the action of $T$ on $C$ is induced by a morphism of cones $T \to C$.

\begin{remark}\label{rmk:present}
If $\mf C \to \mf D$ is a closed embedding of cone stacks on $X$ and $\mf D$ has a global presentation $\mf D = [D/E]$, then $\mf C$ has a global presentation of the form $[C/E]$ where $C = D \times_{\mf D} \mf C$. Indeed, $C$ is a cone on $X$ by \cite[Remark on p.55]{BF}, and the morphism $E \to \mf D$ factors through the vertex $X \to \mf D$, hence through $\mf C \to \mf D$. It follows that $E \to D$ factors through $C$.
\end{remark}

\begin{lemma}\label{lem:independent}
If $[C_1/T_1] = [C_2/T_2]$ are two presentations of the same cone stack, then
\[
c(T) \cap s(C) = c(T') \cap s(C').
\]
\end{lemma}
\begin{proof}
Let $\fX = [C_1/T_1] = [C_2/T_2]$ be a cone stack on $S$. For $(i, j) = (1, 2)$ and $(2, 1)$, there are Cartesian diagrams (with $Z$ defined as the fiber product)
\[
\begin{tikzcd}
Z \arrow[r] \arrow[d] & C_i \arrow[r] \arrow[d] & \fX \arrow[d, "\delta"] \\
C_i \times_S C_j \arrow[r] & {[C_i \times_S C_j/T_j]} \arrow[r] \arrow[d] & \fX \times_S \fX \arrow[d, "pr_1"]\\
& C_i \arrow[r] & \fX
\end{tikzcd}
\]
realizing $[Z/T_2]$ as a presentation for $C_1$ and $[Z/T_1]$ as a presentation for $C_2$. Here $\delta$ is the diagonal and $pr_1$ is projection to the first factor. The result follows from \cite[Lemma 1.3]{BF} and a computation using Lemma \ref{lemma: properties of segre} (cf.~\cite[Example 4.1.6(c)]{fulton}).
\end{proof}

\noindent
Because of Lemma \ref{lem:independent}, the following definition makes sense.
\begin{definition}
If $\mf C \to X$ is a cone stack with a global presentation $\mf C = [C/T]$, then the \emph{Segre class} of $\mf C$ is
\[
s(\mf C) := c(T) \cap s(C) \quad \quad \quad \in \prod A_*(X).
\]

\end{definition}

\subsection{Segre classes of intrinsic normal cones}

Let $f: X \to Y$ be a $DM$-type morphism of algebraic stacks. We say $f$ has a \emph{global presentation} if its intrinsic normal cone has a global presentation (Section \ref{sec:segre1}).

\begin{example}
Let $f: X \to Y$ be a $DM$-type morphism of algebraic stacks. If there is an algebraic stack $M$ and a factorization $X \to M \to Y$ of $f$ with $X \to M$ a closed immersion and $M \to Y$ smooth, then $f$ has a global presentation.
\end{example}

\begin{example}\label{ex:pot}
Let $f: X \to Y$ be a $DM$-type morphism of algebraic stacks. If $\EE \to \LL_{X/Y}$ is a perfect obstruction theory with a global resolution, meaning that $\EE$ is globally isomorphic to a 2-term complex of vector bundles on $X$, then $f$ has a global presentation. We note that by \cite[Lemma 5.0.3]{CJW} this implies many structure morphisms from moduli spaces of maps to moduli spaces of curves have a global presentation.
\end{example}

Let $f: X \to Y$ be a morphism of global presentation and let $(V, E, T)_Y$ be a cycle on $Y$ with morphisms $i_V: V \to E, p_E: E \to T$, and $g_T: T \to Y$. Let $X_T = X \times_{Y} T$, $X_E = X \times_{Y} E$,
        and $X_V = X \times_{Y} V$, with morphisms $i_{X_V}: X_V \to X_E$, $p_{X_E}: X_E \to X_T$, and $g_{X_T}: X_T \to X$. By \cite[Proposition 2.26]{manolache-pullback} we have a closed embedding $C_{X_V} V \to C_X Y$, and from Remark \ref{rmk:present} we have that $C_{X_V} V$ has a global presentation.

\begin{definition}\label{def:segre}

        The \emph{Segre class of $(f; V, E, T)$} is
\[
s(f; V, E, T) := g_{X_T*} (p_{X_E}^*)^{-1} i_{X_V*} s(C_{X_V} V) \quad \quad \quad \in \prod A_*(X).
\]

\end{definition}

\begin{remark}
We do not claim that the mapping $(V, E, T)_Y \mapsto s(f; V, E, T)$ induces a morphism of Chow groups.
\end{remark}

\begin{example}
Let $f: X \to Y$ be a morphism of global presentation. An important special case of Definition \ref{def:segre} is when the cycle $(V, E, T)_Y$ is equal to $(Y, Y, Y)_Y$, namely the fundamental class of $Y.$ In this case,
\[
s(f; Y, Y, Y) = s(C_X Y)
\]
is Fulton's canonical class (\cite[Example 4.2.6]{fulton} and \cite[Definition 4.3]{siebert}).
\end{example}

\subsection{Siebert's formula for virtual pullbacks}

Let $f: X \to Y$ be a DM-type morphism. We say a perfect obstruction theory $\EE \to \LL_{X/Y}$ for $f$ has a \emph{global resolution} if $\EE$ is globally isomorphic to a 2-term complex $[\mls E^{-1} \to \mls E^0]$ of locally free sheaves on $X$. In this case, by Example \ref{ex:pot}, for any cycle $(V, E, T)_Y$ on $Y$ we have a Segre class $s(f; V, E, T)$; moreover the rank of $\EE$ was defined in \cite[Section 5]{BF}. We define $c(\EE^\vee) := c(E^0) \cap c(E^1)^{-1}$ where $E^i = \relSpec_X(Sym (\mls E^{-i}))$.

\begin{theorem}\label{thm:siebert2}
    If $X$ is stratified by global quotient stacks and $\phi \colon \EE \to \LL_{X/Y}$
    is a perfect obstruction theory for $f$ with a global resolution, then
    \[
        f_{\phi}^{!}(\;[(V, E, T)_Y]\;) =
        \left\{ c(\EE^{\vee})^{-1} \cap s(f; V, E, T)\right\}_{\dim(V, E, T) + \rank(\EE)}.
    \]
\end{theorem}
\begin{remark}
The expression $ c(\EE^\vee)^{-1} \cap s(f; V, E, T)$ makes sense because
$s(f; V, E, T)$ has only finitely many nonzero homogeneous components with positive dimension, and each homogeneous component of the operator $c(\EE^\vee)^{-1}$ only decrease dimension (or leave it unchanged). So each homogeneous component of $ c(\EE^\vee)^{-1} \cap s(f; V, E, T)$ is a finite sum.
\end{remark}

\begin{proof}
Let $st(\EE) = [E^1/E^0]$ be the cone stack associated to $\EE$. In this proof, if $F$ is a cone or cone stack on an algebraic stack $M$, we write $p_F: F \to M$ for the structure morphism. If $F$ is a vector bundle stack and $M$ is stratified by global quotient stacks then by \cite{kresch} the flat pullback has an inverse, denoted $(p_F^*)^{-1}$.
\[
\begin{tikzcd}
\widetilde{C}_{X_V} V \arrow[d] \arrow[r] & \widetilde{C}_{X_E} E \arrow[d] \arrow[r] & \widetilde{C}_{X_T} T \arrow[r]\arrow[d]  & \widetilde{C}_X Y \arrow[r, "\widetilde{i}"] \arrow[d, "\widetilde{\pi}"] & E^1\arrow[d, "\pi"]  \\
{C}_{X_V} V \arrow[r] & {C}_{X_E} E \arrow[r] & {C}_{X_T} T \arrow[r] & {C}_X Y \arrow[r, "i"] & st(\EE).
\end{tikzcd}
\]
From the definition of $f^!_\phi$ in \cite[Construction 3.6]{manolache-pullback}, together with the formula for $\sigma$ in Proposition \ref{prop:sigma-on-cycles}, we have
\[
f^!_\phi(\;[(V, E, T)_Y]\;) = (p_\EE^*)^{-1} i_*(\;[(C_{X_V} V, C_{X_E}E, C_{X_T}T)_{C_X Y}]\;).
\]
From compatibility of flat pullbacks and closed embeddings, we have
\[(p_\EE^*)^{-1}i_* = (p_{E^1}^*)^{-1} \pi^*i_* = (p_{E_1}^*)^{-1} \widetilde{i}_*\widetilde{\pi}^*,\]
and hence
\[
f^!_\phi(\;[(V, E, T)_Y]\;) = (p_{E_1}^*)^{-1} \widetilde{i}_*(\;[(\widetilde{C}_{X_V} V, \widetilde{C}_{X_E}E, \widetilde{C}_{X_T}T)_{\widetilde{C}_X Y}]\;).
\]
The right hand side is equal to $(p_{E_1}^*)^{-1} (\;[(\widetilde{C}_{X_V} V, \widetilde{C}_{X_E}E, \widetilde{C}_{X_T}T)_{E^1}]\;),$ where $\widetilde{C}_{X_T} T$ is viewed as a projective stack over $E^1$ via the composition $\widetilde{C}_{X_T} T \to \widetilde{C}_X Y \xrightarrow{\widetilde{i}} E^1.$ Since formation of intrinsic normal cones commutes with flat pullback, the square with $\widetilde{C}_{X_T} T$, $\widetilde{C}_{X_E}E$, $E^1|_{X_T}$, and $E^1|_{X_E}$ is fibered, and by \cite[Definition 2.1.9]{kresch} we have an equality of classes in $A_*(E^1)$
\[[(\widetilde{C}_{X_V} V, \widetilde{C}_{X_E}E, \widetilde{C}_{X_T}T)_{E^1}]= [(\widetilde{C}_{X_V} V, E^1|_{X_E}, E^1|_{X_T})_{E^1}].\]
We have a fibered diagram
\[
\begin{tikzcd}
\widetilde{C}_{X_V} V \arrow[r] & E^1|_{X_V} \arrow[d] \arrow[r, "\widetilde{i}"] & E^1|_{X_E} \arrow[d] \arrow[r, "\widetilde{p}"] & E^1|_{X_T} \arrow[d] \arrow[r, "\widetilde{g}"] & E^1 \arrow[d] \\
& X_V \arrow[r, "i_{X_V}"] & X_E \arrow[r, "p_{X_E}"] & X_T \arrow[r, "g_{X_T}"] & X.
\end{tikzcd}
\]
By Fact \ref{fact:cycle-is} we have
\[
f^!_\phi(\;[(V, E, T)_Y]\;) = (p_{E^1}^{\ast})^{-1}\widetilde{g}_*(\widetilde p ^*)^{-1}\widetilde{i}_*[\widetilde C_{X_V} V]
\]
and by compatibility of flat pullbacks and proper pushforwards, this is equal to
\[
g_{X_T*}(p_{X_E}^*)^{-1} i_{X_V*}(p_{E^1|_{X_V}}^*)^{-1}[\widetilde{C}_{X_V} V].
\]
We now apply \cite[Example 4.1.8]{fulton} (it is straightforward to check that this result holds in our setting by the same proof) and compute
\begin{align*}
f^!_\phi(\;[(V, E, T)_Y]\;) &= g_{X_T*}(p_{X_E}^*)^{-1} i_{X_V*}\{c(E^1|_{X_V}) \cap s(\widetilde{C}_{X_V} V)\}_{\dim \widetilde{C}_{X_V} V - \rank(E^1)}\\
&= \{g_{X_T*}(p_{X_E}^*)^{-1} i_{X_V*}c(E^1|_{X_V}) \cap s(\widetilde{C}_{X_V} V)\}_{\dim (V, E, T) + \rank(\mathbb{E})}\\
&= \{(c(E^1) \cup c(E^0)^{-1}) \cap g_{X_T*}(p_{X_E}^*)^{-1} i_{X_V*}( c(E^0|_{X_V}) \cap s(\widetilde{C}_{X_V} V) )\}_{\dim (V, E, T) + \rank(\mathbb{E})}
\end{align*}
Here, the second equality uses $\dim \widetilde{C}_{X_V} V = \dim V + \rank(E^0)$
and the third uses functoriality of Chern classes. The theorem now follows from the definitions of $c(\EE^\vee)^{-1}$ and $s(f; V, E, T).$

\end{proof}

\section{Application to Gromov-Witten theory}
In this section we use Theorem \ref{thm:siebert2} to (re)prove the Splitting Axiom (\cite[Def 7.1(3)]{BM96} and \cite[Axiom III]{BehrendGW}).

\subsection{Moduli of sections}\label{sec:setup}

Let $ \fM$ be a locally  Noetherian algebraic stack, let $C \to \fM$ be a family of tame twisted curves as in \cite[Def 2.1]{AOV}, and let $ Z \to C$ be a morphism of algebraic stacks such that $Z \to  \fM$ is locally of finite presentation, quasi-separated, and has affine stabilizers. By \cite[Thm 1.3]{HR19} there is an algebraic stack $S_\fM(Z/C)$ over $\fM$ whose groupoid fiber over a scheme $T \to \fM$ is $\Hom_C (C \times_\fM T, Z)$. By \cite[Thm 1.1]{webb} this stack carries a canonical obstruction theory, defined as follows.

Let $\pi: C_S \to S_\fM(Z/C)$ be the universal curve and let $f: C_S \to Z$ be the universal section. There is a relative dualizing complex $\omega^\bullet_{C_S/S} = \omega_{C_S/S}[1]$ for the family $C_S \to S_\fM(Z/C)$ and a trace map $tr_{C_S/S}: R\pi_*\omega^\bullet_{C_S/S} \to \mls O_{S_\fM(Z/C)}$, and these are compatible with base change  \cite[Prop 3.14]{webb}.  Let $\Dqc(S)$ be the full subcategory of the unbounded derived category of $\mls O_S$-modules in the lisse-\'etale topology on objects with quasi-coherent cohomology sheaves.

\begin{definition}\label{def:a}
Define $\pi_!$ to be the functor $R\pi_*(- \otimes \omega_{C_S/S}^\bullet)$. Define the \emph{adjunction-like map} $a_\pi$ from $\Hom_{\Dqc(C_S)}(\mls F, \pi^*\mls G)$ to $\Hom_{\Dqc(S)}(\pi_! \mls F, \mls G)$ to send $\mls F \to \pi^*\mls G$ to the composition
\[
R\pi_*(\mls F \otimes \omega_{C_S/S}^\bullet) \to R\pi_*(\pi^*\mls G \otimes \omega_{C_S/S}^\bullet) \xleftarrow{\sim} \mls G \otimes R\pi_*(\omega_{C_S/S}^\bullet) \xrightarrow{tr_{C_S/S}} \mls G \otimes \mls O_S \xrightarrow{\sim} \mls G
\]
where the backwards isomorphism is the projection formula (cf. \cite[Section 2.3]{webb}).
\end{definition}

\begin{remark}\label{rmk:why-glue}
When $S$ is a separated Noetherian tame algebraic stack, the dualizing complex $\omega^\bullet_{C_S/S}$ is equal to $\pi^! \mls O_S$ where $\pi^!$ is right adjoint to $R\pi_*$ and $tr_{C_S/S}$ is the counit of the $(R\pi_*,  \pi^!)$ adjunction. Moreover $\pi_!$ is a genuine left adjoint to $\pi^*$ and $a_\pi$ is the adjunction bijection \cite[Example 2.11]{webb}.  In general we do not know if these identifications hold because we do not know if the right adjoint to $R\pi_*$ is compatible with base change: the base change result \cite[Lem 3.7]{webb} requires the derived category $\Dqc(S)$ to be compactly generated. If $S$ is for example a moduli stack of twisted stable maps to a tame Artin stack over a field of positive characteristic, then we don't know if this hypothesis holds for $S$.
\end{remark}

There are canonical morphisms of cotangent complexes
\begin{equation}\label{eq:OT1}
f^*\LL_{Z/C}  \to \LL_{C_S/C} \xleftarrow{\sim} \pi^*\LL_{S_\fM(Z/C)/\fM}.
\end{equation}
Applying the adjunction-like morphism $a_\pi$ to \eqref{eq:OT1} yields the morphism
\begin{equation}\label{eq:OT2}
\phi_{S/\fM}\;: \quad  \quad \EE_{S/\fM} := R\pi_*(f^*\LL_{Z/C} \otimes  \omega^\bullet_{C_S/S}) \to \LL_{S_\fM(Z/C)/\fM}.
\end{equation}
By \cite[Thm 1.1]{webb} the morphism \eqref{eq:OT2} is a relative obstruction theory for the morphism $S_\fM(Z/C) \to \fM$.

\subsection{Relative duality for morphisms of curves}

To state our functoriality lemma, we need the following result (cf. \cite[Prop 3.14]{webb}).

\begin{lemma}\label{lem:rel-duality}
Let $M$ be a locally Noetherian algebraic stack, and let $Q: C_M \to D_M$  be a morphism of tame twisted curves over $M$ such that $R^1Q_*\omega_{C_M/M} = 0$. There is a line  bundle $\nu_{C_M/D_M}$ on $C_M$ and a morphism $tr_{{C_M/D_M}}: RQ_*\nu_{C_M/D_M} \to \mls O_{D_M}$ such that
\begin{itemize}
\item[(i)] The pair is functorial in the following sense: given a \Rachel{flat}  morphism $m: N\to M$ of locally Noetherian algebraic stacks, with $C_N = C_M \times_M N$ and $D_N = D_M \times_M  N$ and $m': C_N \to C_M$, $m'': D_N \to D_M$, and $Q_N: C_N \to D_N$  the canonical maps, there is an isomorphism $m'^* \nu_{C_M/D_M} \to \nu_{C_N/D_N}$ functorial in $N \to M$ such that the square
\[
\begin{tikzcd}[column sep = 40]
Lm''^*RQ_*\nu_{C_M/D_M} \arrow[r, "{Lm''^* tr_{C_M/D_M}}"] \arrow[d, "\sim"'] & \mls O_{D_N}\\
RQ_{N*} Lm'^*\nu_{C_M/D_M} \arrow[r, "\sim"] & RQ_{N*} \nu_{C_N/D_N} \arrow[u, "tr_{C_N/D_N}"']
\end{tikzcd}
\]
commutes, where the left vertical arrow is the base change isomorphism and the bottom is induced by the isomorphism $m'^* \nu_{C_M/D_M} \to \nu_{C_N/D_N}$.

\item[(ii)] If $M$ is a quasi-separated Noetherian algebraic space, then $\nu_{C_M/D_M} = Q^! \mls O_{D_M}$ and $tr_{C_M/D_M}$ is the counit of the $(RQ_*, Q^!)$ adjunction. (This adjoint exists by  e.g. \cite[Thm 1.3]{BDS}.)
\end{itemize}
\end{lemma}
\begin{proof}
The proof of \cite[Prop 3.14]{webb} works here. There are two things to check.
First, we need to know that when $N \to M$ is a flat morphism of quasi-separated Noetherian algebraic spaces, the canonical map $m'^*\nu_{C_M/D_M} \to \nu_{C_N/D_N}$ is an isomorphism. For this we use \cite[Lemma 0.1]{neeman}, noting that $Q$ is concentrated (by \cite[Lemma~3.3(2)]{webb} and
\cite[Lemma~2.5(4)]{hall-rydh}) and proper.
Second, we must check that the Ext vanishing required for gluing the trace maps still holds. Observe that $\nu_{C_M/D_M}$ is equal to $\omega_{C_M/M} \otimes Q^*\omega_{D_M/M}^{-1},$ so $R^1Q_*\nu_{C_M/D_M}$ vanishes because $R^1Q_*\omega_{C_M/M}$ does by hypothesis.
\end{proof}

Mimicking Definition \ref{def:a}, we define a functor $Q_!:= RQ_*(- \otimes \nu_{C_M/D_M})$ and an adjunction-like map $a_Q$ from $\Hom_{\Dqc(C_M)}(\mls F, LQ^* \mls G)$ to $\Hom_{\Dqc(D_M)}(Q_!\mls F, \mls G)$: explicitly, this map is to apply $Q_!$ and then compose with the projection isomorphism and $tr_{C_M/D_M}$.

\begin{remark}\label{rmk:why-glue2} As in Remark \ref{rmk:why-glue},
if $M$ is a separated Noetherian tame algebraic stack, one can construct $\nu_{C_M/D_M}$ and $tr_{C_M/D_M}$ globally, via adjunction, rather than using the gluing approach of Lemma \ref{lem:rel-duality}. By \cite[Example 2.11]{webb} the functor $Q_!$ and the adjunction-like map $a_Q$ will agree with the analogous data arising from adjunction.
\end{remark}

\subsection{Functoriality in the curve}

The goal of this section is to prove a functoriality statement for \eqref{eq:OT2} ``in the curve'' $C$. Our setup is the following.

Let $\fM$ and $\mf N$ be locally Noetherian algebraic stacks and let $D \to \mf N$, $C \to \fM$ be two families of tame twisted curves. Let $\fp: \mf N \to \mf M$ be a morphism, let $C_{\mf N} = C \times_{\mf M} \mf N$, and let $Q: C_{\mf N} \to D$ be a morphism of stacks over $\mf N$.
Let $Z_D \to D$ and $Z_C \to C$ be flat morphisms of algebraic stacks satisfying the conditions of Section \ref{sec:setup}, and assume the pullbacks $Z_C \times_C C_{\mf N}$ and $Z_D \times_D C_{\mf N}$ are equal. (For an example of how to apply this setup, see Section \ref{sec:splitting}.) Let $M = S_\fM(Z_C/C)$ and $N = S_{\mf N}(Z_D/D)$ be the respective moduli of sections, with universal curves
\[
C_{M} = C \times_{\fM} M \quad \quad \text{and} \quad \quad D_N = D  \times_{\mf N} N
\]
having projection morphisms $\pi_M: C_{M} \to  M$ and $\pi_N: D_N \to  N$ and universal sections $f_M: C_M \to Z_C$ and $f_N: D_N \to Z_N$, respectively. Let $C_N = C_{\mf N} \times_{\mf N} N$, and let $p: N \to M$ be the morphism induced by the composition $C_N \to D_N \xrightarrow{f_D} Z_D$. These spaces and morphisms are summarized in the following diagram.
\[\begin{tikzcd}[row sep=10]
& Z_C && W && Z_D \\
& C && C_{\mf N} && D \\
C_M && C_N && D_N \\
& \mf M && \mf N \\
M && N
\arrow[from=1-2, to=1-4]
\arrow[from=1-2, to=2-2]
\arrow[from=1-4, to=1-6]
\arrow[from=1-4, to=2-4]
\arrow[from=1-6, to=2-6]
\arrow[from=2-2, to=4-2]
\arrow["{{\mathfrak{p}_C}}"'{pos=0.7}, from=2-4, to=2-2]
\arrow["Q"{pos=0.3}, from=2-4, to=2-6]
\arrow[curve={height=-12pt}, from=2-6, to=4-4]
\arrow["f_C"{pos=0.7}, curve={height=-12pt}, from=3-1, to=1-2]
\arrow[from=3-1, to=2-2]
\arrow["{\pi_{C/M}}"'{pos=0.8}, from=3-1, to=5-1]
\arrow["f"{pos=0.7}, curve={height=-12pt}, from=3-3, to=1-4, crossing over]
\arrow[from=3-3, to=2-4]
\arrow["{p_{C}}"{pos=0.3}, from=3-3, to=3-1]
\arrow["f_D"{pos=0.7}, curve={height=-12pt}, from=3-5, to=1-6, crossing over]
\arrow[from=3-5, to=2-6]
\arrow["{\pi_{D/N}}", curve={height=-12pt}, from=3-5, to=5-3, crossing over]
\arrow["{\mathfrak{p}}"'{pos=0.8}, from=4-4, to=4-2]
    \arrow["{\pi_{C/N}}"'{pos=0.8}, from=3-3, to=5-3, crossing over]
\arrow[from=4-4, to=2-4]
\arrow[from=5-1, to=4-2]
\arrow[from=5-3, to=4-4]
\arrow["p", from=5-3, to=5-1]
    \arrow["Q_N"{pos=0.8}, from=3-3, to=3-5, crossing over]
\end{tikzcd}\]

\begin{lemma}\label{lem:key-lemma}
Assume $R^1Q_*\omega_{C_\mf N/\mf N} = 0,$ and let $\mathbb{K}$ denote the mapping cone of $\mls O_{D_N} \to RQ_{N*} \mls O_{C_N}$ with $\mathbb{K}^\vee := R\Hom(\mathbb{K}, \mls O_{D_N})$. There is a commuting square
\begin{equation}\label{eq:key-lemma}
\begin{tikzcd}
Lp^*\EE_{M/\mf M} \arrow[r] \arrow[d, "Lp^*\phi_{M/\mf M}"'] & \EE_{N/\mf N} \arrow[d, "\phi_{N/\mf N}"] \\
Lp^* \LL_{M/\mf M} \arrow[r] & \LL_{N/\mf N}
\end{tikzcd}
\end{equation}
where the bottom morphism is the canonical morphism of cotangent complexes and the
mapping cone of the top arrow is
\[
R\pi_{D/N*}(Lf_D^*\LL_{Z_D/D} \otimes \omega^\bullet_{D/N} \otimes \mathbb{K}^\vee[1]).
\]
\end{lemma}
\begin{proof}
The assumption that $R^1Q_*\omega_{C_\mf N/\mf N} = 0$ is needed so that we can define $(\nu_{C_{\mf N}/D}, tr_{C_{\mf N}/D})$ using Lemma \ref{lem:rel-duality}.
To construct \eqref{eq:key-lemma}, start with the commuting diagram of canonical morphisms of cotangent complexes
\[
\begin{tikzcd}[row sep=10]
Lp_C^*Lf_C^*\LL_{Z_C/C} \arrow[d] \arrow[r, "\sim"] & Lf^*\LL_{W/C_\mf N} \arrow[d] &\arrow[l, "\sim"'] LQ_N^*Lf_D^*\LL_{Z_D/D} \arrow[d] \\
Lp_C^*\LL_{C_M/C} \arrow[r] & \LL_{C_N/C_{\mf N}} & LQ_N^* \LL_{D_N/D} \arrow[l, "\sim"']\\
Lp_C^*\pi_{C/M}^*\LL_{M/\fM} \arrow[u, "\sim"] \arrow[rr] \arrow[d, equal]&& LQ_N^*\pi_{D/N}^* \LL_{N/\mf N} \arrow[u, "\sim"'] \arrow[d, equal]\\
LQ_N^*\pi_{D/N}^*Lp^*\LL_{M/\mf M} \arrow[rr] && LQ_N^*\pi_{D/N}^* \LL_{N/\mf N}
\end{tikzcd}
\]
The top horizontal arrows are isomorphisms since $Z_C \to C$ and $Z_D \to D$ are flat, and the remaining horizontal arrow is an isomorphism since both $C_{\mf N}$ and $D$ are flat over $\mf N$. Observe that this is a diagram of complexes on $C_N$.
Applying the adjunction-like maps $a_{Q_N}$ and then $a_{\pi_{D/N}}$ to the vertical columns yields the (commuting) square in the diagram
\begin{equation}\label{eq:key-lemma2}
\begin{tikzcd}[column sep=20]
\pi_{D/N!} Q_{N!} Lp_C^*Lf_C^*\LL_{Z_C/C} \arrow[d] \arrow[r, "\sim"] & \pi_{D/N!} Q_{N!} LQ_N^* Lf_D^*\LL_{Z_D/D} \arrow[d] \arrow[r] & \pi_{D/N!} Lf_D^*\LL_{Z_D/D} \arrow[dl]\\
Lp^*\LL_{M/\mf M} \arrow[r] & \LL_{N/\mf N}.
\end{tikzcd}
\end{equation}
Recalling that $Q_{N!}$ is the functor $RQ_{N*}(- \otimes \nu_{C_N/D_N})$, we define the top  right horizontal arrow to be induced by a projection formula isomorphism $RQ_{N*}(LQ_N^*Lf_D^*\LL_{Z_D/D} \otimes \nu_{C_N/D_N}) \simeq Lf_D^*\LL_{Z_D/D} \otimes Q_{N*}\nu_{C_N/D_N}$ and $tr_{C_N/D_N}$.
The diagonal arrow is $\phi_{N/\mf N}$, and the left vertical arrow is isomorphic to $Lp^*\phi_{M/\mf M}$ by compatibility of the adjunction-like maps $a_{\pi_{C/M}}$ and $a_{\pi_{C/N}}$ and base change \cite[Lem 2.13]{webb}. So \eqref{eq:key-lemma2} is isomorphic to the desired diagram \eqref{eq:key-lemma}.

To show that the triangle in \eqref{eq:key-lemma2} commutes,
we can work locally on $N$, and hence replace $N$ with an affine scheme. In this case, by Remarks \ref{rmk:why-glue} and \ref{rmk:why-glue2} the dualizing sheaves $\omega_{C_N/N},$ $ \omega_{D_N/N},$ $ \nu_{C_N/D_N}$ and associated trace maps, lower shriek functors, and adjunction-like maps all agree with the analogous objects defined using adjunction. So commutativity of the triangle follows from the statement that the composition of adjoints of $Q_N^*$ and $\pi_{D/N}^*$ is adjoint to the composition $ Q_N^* \circ \pi_{D/N}^* = \pi_{C/N}^*$. This completes the proof that \eqref{eq:key-lemma} commutes.

It remains to compute the mapping cone of the top arrow. By assumption we have a distinguished triangle
\[
\mls O_{D_N} \to RQ_{N*} \mls O_{C_N} \to \mathbb{K} \xrightarrow{+1}.
\]
Dualizing gives another distinguished triangle
\begin{equation}\label{eq:dt}
R\Hom(\mathbb{K}, \mls O_{D_N}) \to R\Hom(RQ_{N*} \mls O_{C_N}, \mls O_{D_N}) \to \mls O_{D_N} \xrightarrow{+1}
\end{equation}
where $R\Hom$ denotes internal hom for $\Dqc(D_N)$. We claim the arrow $R\Hom(RQ_{N*} \mls O_{C_N}, \mls O_{D_N}) \to \mls O_{D_N}$ is isomorphic to $tr_{{C_N/D_N}}:\nu_{C_N/D_N} \to  \mls O_{D_N}$. Globally, the isomorphism is given by a commuting diagram
\begin{equation}\label{eq:key-lemma3}
\begin{tikzcd}
R\Hom(RQ_{N*} \mls O_{C_N}, \mls O_{D_N}) \arrow[r]  &  R\Hom(\mls O_{D_N}, \mls O_{D_N})\\
RQ_{N*} R\Hom(\mls O_{C_N}, \nu_{C_N/D_N}) \arrow[u, "\sim"] &  \arrow[l, "\sim"'] R\Hom(\mls O_{D_N}, RQ_{N*} \nu_{C_N/D_N}) \arrow[u, "tr_{C_N/D_N}"]
\end{tikzcd}
\end{equation}
where the left vertical arrow is defined  using the trace map as in \cite[Prop 4.4]{FHM} and the bottom horizontal arrow is \cite[(3.4)]{FHM}. We can check locally that this diagram commutes (and that the left vertical arrow is an isomorphism); that is, we can check these things after replacing $N$ by an affine scheme. In this case all dualizing sheaves and trace maps arise formally from adjunction, and we can apply Proposition \ref{prop:unit-counit}. The internal adjunction isomorphisms in \eqref{eq:key-lemma3}  and \eqref{eq:key-lemma4} agree by \cite[Prop 4.4]{FHM} and \cite[111]{FHM}.
Tensoring the triangle \eqref{eq:dt} with $Lf_D^*\LL_{Z_D/D} \otimes \omega^\bullet_{D/N}$ and applying $R\pi_{D/N*}$ yields the desired triangle
\[
\pi_{D/N!} Q_{N!} LQ_N^* Lf_D^* \LL_{Z_D/D} \to \pi_{D/N!} Lf_D^* \LL_{Z_D/D} \to R\pi_{D/N*}(Lf_D^* \LL_{Z_D/D} \otimes \omega^\bullet_{D/N} \otimes \mathbb{K}^\vee[1]) \xrightarrow{+1}.
\]
\end{proof}

\subsection{Application to the splitting axiom}\label{sec:splitting}

In this section we  work over the complex numbers.
Let $X$ be a smooth complex projective variety. Let $\tau, \sigma$ be $H_2(X)^+$-marked stable graphs (\cite[Defs 1.6, 1.9]{BM96})
such that $\tau$ is obtained from $\sigma$ by splitting an edge,
and let $\mf{M}_{\tau}$ and
$\mf{M}_{\sigma}$ be the associated moduli stacks of prestable marked curves with universal curves $C_\tau \to \mf M_\tau$ and $C_\sigma \to \mf M_\sigma$, respectively. The moduli stacks $\mf{M}_{\tau}$,
$\mf{M}_{\sigma}$ of prestable marked curves are isomorphic, so we denote both stacks by $\mf M$ (but $C_\tau$ and $C_\sigma$ are not isomorphic). Let $\mc M(X, \sigma)$ and $\mc M(X, \tau)$ be the moduli stacks of stable maps to $X$ of type $\sigma$ and $\tau$, respectively. There is a cartesian diagram
\[
    \begin{tikzcd}[ampersand replacement = \&]
        \mc{M}(X, \sigma) \arrow[r, "p"] \arrow[d, ] \& \mc{M}(X, \tau) \arrow[d] \\
        \mf{M} \times X \arrow[r, "{\id \times \Delta}"] \& \mf{M} \times X \times X.
    \end{tikzcd}
\]
where vertical arrows are induced by evaluation at the markings corresponding to the (cut) edge in $\sigma$ and $\tau$, the top horizontal arrow sends the universal stable map $C_\sigma \to X$ to the composition $C_\tau \to C_\sigma \to X$, and the bottom horizontal arrow is induced by the diagonal $\Delta: X \to X \times X$.

\begin{theorem}[Splitting Axiom, {\cite[Def 7.1(3)]{BM96} and \cite[Axiom III]{BehrendGW}}]
There is an equality
\[
[\mc M(X, \sigma)]^{\vir} = (\id \times \Delta)^! [\mc M(X, \tau)]^{\vir}.
\]
\end{theorem}
\begin{proof}
We apply Lemma \ref{lem:key-lemma}
with $\mf M = \mf N$ the moduli of prestable $\tau$-marked (or $\sigma$-marked) curves, $C = C_\tau$, $D = C_\sigma$, $Z_C = C_\tau \times X$, $Z_D = C_\sigma \times X$, $M = \mc M(X, \tau)$, and  $N = \mc M(X, \sigma)$. Let let $\pi: C_\sigma \to \mc M(X, \sigma)$ denote the projection, let $f: C_\sigma \to X$ be the universal map, and let $n: \mc M(X, \sigma) \to C_\sigma$ denote the section whose image is the node split by $Q: C_\tau \to C_\sigma$. Note that if $\mls F$ is a quasicoherent sheaf on $C_\tau$ then $R^1Q_*\mls F=0$ since $Q$ is affine. We have a short exact sequence
\[
0 \to \mls O_{C_\sigma} \to Q_{*} \mls O_{C_\tau} \to n_*\mls O_{\mc M(X, \sigma)} \to 0.
\]

The  Lemma gives us a commuting square
\begin{equation}\label{eq:square}
\begin{tikzcd}
Lp^*\EE_{\mc M(X, \tau)/\mf M} \arrow[r] \arrow[d, "Lp^* \phi_{\mc M(X, \tau)/\mf M}"'] & \EE_{\mc M(X, \sigma)/\mf M} \arrow[d, "\phi_{\mc M(X, \sigma)/\mf M}"] \\
Lp^* \LL_{\mc M(X, \tau)/\mf M} \arrow[r] & \LL_{\mc M(X, \sigma)/\mf M}
\end{tikzcd}
\end{equation}
such that the mapping cone of the top horizontal arrow is
\begin{equation}\label{eq:split2}
R\pi_*(f^*\Omega_X \otimes \omega^\bullet_{C_\sigma/\mc M(X, \sigma)} \otimes (n_*\mls O_{\mc M(X, \sigma)})^\vee[1]).\end{equation}
Let $f_n = f \circ  n$. We compute
\begin{align*}
\text{\eqref{eq:split2}} &=R\pi_*R \Hom(\, f^*\Omega_X^\vee \otimes (\omega^\bullet_{C_\sigma/\mc M(X, \sigma)})^\vee,\; (n_*\mls O_{\mc M(X, \sigma)})^\vee[1] \, )\\
&=  R\pi_*R \Hom(\, f^*\Omega_X^\vee \otimes (\omega^\bullet_{C_\sigma/\mc M(X, \sigma)})^\vee \otimes n_*\mls O_{\mc M(X, \sigma)}[-1],\; \mls O_{C_\sigma} \, )\\
&= R\pi_*R \Hom(\, f^*\Omega_X^\vee \otimes  n_*\mls O_{\mc M(X, \sigma)}[-1],\; R\Hom( (\omega^\bullet_{C_\sigma/\mc M(X, \sigma)})^\vee ,\mls O_{C_\sigma}) \, )\\
&= R\pi_*R \Hom(\, f^*\Omega_X^\vee \otimes  n_*\mls O_{\mc M(X, \sigma)}[-1],\; \omega^\bullet_{C_\sigma/\mc M(X, \sigma)} \, )\\
&= (R\pi_*n_*(f_n^*\Omega_X^\vee[-1]))^\vee\\
&= f_n^*\Omega_X[1].
\end{align*}
In order, the equalities use \cite[Tag 08JJ]{stacks-project}, \cite[Tag 08J9]{stacks-project}, \cite[Tag 08J9]{stacks-project}, \cite[Tag 08JJ]{stacks-project}, Grothendieck duality, and the fact that $\pi \circ n$ is the identity.

Let $\FF$ denote the complex $f_n^*\Omega_X[1]$. It  follows from \eqref{eq:square} that there is a compatible triple $(\FF, \, \EE_{\mc M(X, \tau)/\fM}, \,\EE_{\mc M(X, \sigma)/\fM})$ where $\phi: \FF \to \LL_{\mc M(X, \sigma)/\mc M(X, \tau)}$ is induced by the mapping cone axiom. Hence by \cite[Cor 4.9]{manolache-pullback}, we have
\[
p_{\phi}^! [\mc M(X, \tau)]^\vir = [\mc M(X, \sigma)]^{\vir}.
\]
On the other hand, by Theorem \ref{thm:siebert2}, the virtual pullback $p^!_\phi$ is determined by $\FF$ and is independent of $\phi$. The cone stack associated to $\FF$ is precisely the pullback of the normal cone to $X \to X \times X$ along the nodal section,
so we have $p_\phi^! = (\mathrm{id} \times \Delta)^!$.

\end{proof}

\appendix
\section{Dualizing the trace map}

Let $\mls C$ and $\mls D$ be closed symmetric monoidal categories, with $f^*: \mls D \to \mls C$ a strong symmetric monoidal functor and $f_*$ its right adjoint. Assume we have also a right adjoint $f^!$ to $f_*$ and that the projection formula holds, meaning that the adjoint of $f^*(x \otimes  f_*(y)) \simeq f^*(x) \otimes f^*f_*(y) \to f^*(x) \otimes y$ is an isomorphism. Let $\mls O_{\mls D}$ denote the unit object of $\mls D$ and let $\Hom_{\mls D}(-, -)$ denote internal Hom.
Proposition \ref{prop:unit-counit} says that the dual of the unit $\eta: \mls O_{\mls D} \to f_*f^* \mls O_{\mls D}$  is canonically identified with the counit $\epsilon: f_*f^! \mls O_{\mls D} \to \mls O_{\mls D}.$

\begin{proposition}\label{prop:unit-counit}
There is a commuting diagram
\begin{equation}\label{eq:key-lemma4}
\begin{tikzcd}
\Hom_{\mls D}(f_*f^*\mls O_{\mls D},  \mls O_{\mls D}) \arrow[r, "- \circ \eta"] \arrow[d, leftrightarrow, "\sim"']  & \Hom_{\mls D}(\mls O_{\mls D}, \mls O_{\mls D})\\
f_*\Hom_{\mls D}(f^*\mls O_{\mls D, f^! \mls  O_{\mls D}}) \arrow[r, leftrightarrow, "\sim"] & \Hom_{\mls D}(\mls O_{\mls D}, f_*f^!\mls O_{\mls D}) \arrow[u, "\epsilon \circ -"']
\end{tikzcd}
\end{equation}
where the left isomorphism is \cite[(2.18)]{BDS} and the bottom one  is \cite[(2.17)]{BDS}.
\end{proposition}
The proof occupies the remainder of the section. For objects $X, Y \in \mls D$ let $\mls D(X, Y)$ denote the set of morphisms from $X$ to $Y$. To lighten notation let $\mls O = \mls O_{\mls D}$ and  let $\Hom = \Hom_{\mls D}$. By the Yoneda lemma, it is enough to demonstrate commutativity of
\[\begin{tikzcd}[row sep=12]
	{{}^{\gray{1}}\mathscr{D}(X, \mathrm{Hom}(f_*f^*\mathscr{O}, \mathscr{O}))} && {\mathscr{D}(X, \mathrm{Hom}(\mathscr{O}, \mathscr{O})){}^{\gray{3}}} \\
	{\mathscr{D}(X,f_*\mathrm{Hom}(f^*\mathscr{O}, f^!\mathscr{O}))} \\
	{{}^{\gray{2}}\mathscr{D}(X, \mathrm{Hom}(\mathscr{O},f_*f^!\mathscr{O})) }
	\arrow["{\mathrm{(id, -\circ\eta)}}", from=1-1, to=1-3]
	\arrow["\sim"', tail reversed, from=1-1, to=2-1]
	\arrow["\sim"', tail reversed, from=2-1, to=3-1]
	\arrow["{\mathrm{(id, \epsilon \circ-)}}"', from=3-1, to=1-3].
\end{tikzcd}\]
We will do this by building a triangular prism of diagrams with the desired one as its base.
Light gray numbers/letters indicate how to join the diagrams together, equalities denote adjunction bijections, $pr$ denotes projection isomorphisms, and $\lambda$ denotes isomorphisms witnessing compatibility of $f^*$ and $\otimes$.
We now list the additional diagrams. The diagram
\[\begin{tikzcd}[column sep=5pt, row sep=12]
	{{}^{\gray{1}}\mathscr{D}(X, \mathrm{Hom}(f_*f^*\mathscr{O}, \mathscr{O}))} &&& {\mathscr{D}(X\otimes f_*f^*\mathscr{O}, \mathscr{O}){}^{\gray{a}}} \\
	{\mathscr{D}(X,f_*\mathrm{Hom}(f^*\mathscr{O}, f^!\mathscr{O}))} & {\mathscr{D}(f^*X, \mathrm{Hom}(f^*\mathscr{O}, f^!\mathscr{O}))} & {\mathscr{D}(f^*X\otimes f^*\mathscr{O}, f^!\mathscr{O})} & {\mathscr{D}(f_*(f^*X\otimes f^*\mathscr{O}), \mathscr{O})} \\
	{{}^{\gray{2}}\mathscr{D}(X,\mathrm{Hom}(\mathscr{O},f_*f^!\mathscr{O}))} & {\mathscr{D}(X\otimes\mathscr{O}, f_*f^:\mathscr{O})} & {\mathscr{D}(f^*(X\otimes\mathscr{O}), f^!\mathscr{O}){}^{\gray{b}}}
	\arrow[equals, from=1-1, to=1-4]
	\arrow["\sim"', tail reversed, from=1-1, to=2-1]
	\arrow[equals, from=2-1, to=2-2]
	\arrow["\sim"', tail reversed, from=2-1, to=3-1]
	\arrow[equals, from=2-2, to=2-3]
	\arrow[equals, from=2-3, to=2-4]
	\arrow["{(\mathrm{pr, id})}"', from=2-4, to=1-4]
	\arrow[equals, from=3-1, to=3-2]
	\arrow[equals, from=3-2, to=3-3]
	\arrow["{(\lambda, \mathrm{id})}"', from=3-3, to=2-3]
\end{tikzcd}\]
commutes by definition of the isomorphisms \cite[(2.18)]{BDS} and \cite[(2.17)]{BDS}. The diagram
\[\begin{tikzcd}[column sep = 40pt, row sep=12]
	{{}^{\gray{1}}\mathscr{D}(X, \mathrm{Hom}(f_*f^*\mathscr{O}, \mathscr{O}))} & {\mathscr{D}(X,\mathrm{Hom}(\mathscr{O}, \mathscr{O})){}^{\gray{3}}} \\
	{{}^{\gray{a}}\mathscr{D}(X\otimes f_*f^*\mathscr{O}, \mathscr{O})} & {\mathscr{D}(X\otimes\mathscr{O},\mathscr{O}){}^{\gray{d}}}
	\arrow["{(\mathrm{id}, -\circ\eta)}", shift left, from=1-1, to=1-2]
	\arrow[equals, from=1-1, to=2-1]
	\arrow["{\mathrm{id}\otimes\eta, \mathrm{id}}"', from=2-1, to=2-2]
	\arrow[equals, from=2-2, to=1-2]
\end{tikzcd}\]
commutes by naturality of the tensor-Hom adjunction in the object.
The top cell in the diagram
\[\begin{tikzcd}[row sep=12]
{{}^{\gray{a}}\mathscr{D}(X\otimes f_*f^*\mathscr{O}, \mathscr{O})} && {\mathscr{D}(X\otimes\mathscr{O}, \mathscr{O})} \\
{\mathscr{D}(f_*(f^*X\otimes f^*\mathscr{O}), \mathscr{O})} && {\mathscr{D}(f_*f^*(X\otimes\mathscr{O}), \mathscr{O})^{\gray{c}}} \\
{\mathscr{D}(f^*X\otimes f^*\mathscr{O}, f^:\mathscr{O})} \\
{{}^{\gray{b}}\mathscr{D}(f^*(X\otimes\mathscr{O}), f^:\mathscr{O})}
\arrow["{(\mathrm{id}\otimes\eta, \mathrm{id})}", shift left, from=1-1, to=1-3]
\arrow["{\mathrm{(pr, id)}}", from=2-1, to=1-1]
\arrow[equals, from=2-1, to=3-1]
\arrow["{(\eta_{X\otimes\mathscr{O}}, \mathrm{id})}"', from=2-3, to=1-3]
\arrow["{(f_*\lambda, \mathrm{id})}"', from=2-3, to=2-1]
\arrow[equals, from=4-1, to=2-3]
\arrow["{(\lambda, \mathrm{id})}", from=4-1, to=3-1]
\end{tikzcd}\]
commutes by \cite[{3.4.7(i)}]{lipman}, and the bottom cell is functoriality of the $(f_*, f^!)$  adjunction.
Finally, the bottom cell in the diagram
\[\begin{tikzcd}[column sep=40, row sep=12]
{{}^{\gray{b}}\mathscr{D}(f^*(X\otimes\mathscr{O}), f^:\mathscr{O})} & {\mathscr{D}(f_*f^*(X\otimes\mathscr{O}), \mathscr{O})){}^{\gray{c}}} \\
{\mathscr{D}(X\otimes\mathscr{O}, f_*f^:\mathscr{O})} & {\mathscr{D}(X\otimes\mathscr{O}, \mathscr{O}){}^{\gray{d}}} \\
{{}^{\gray{2}}\mathscr{D}(X,\mathrm{Hom}(\mathscr{O}, f_*f^:\mathscr{O}))} & {\mathscr{D}(X, \mathrm{Hom}(\mathscr{O}, \mathscr{O})){}^{\gray{3}}}
\arrow[equals, from=1-1, to=1-2]
\arrow[equals, from=1-1, to=2-1]
\arrow["{(\eta_X{}_\otimes{}_\mathscr{O}, \mathrm{id})}", from=1-2, to=2-2]
\arrow["{(\mathrm{id}, \eta)}", from=2-1, to=2-2]
\arrow[equals, from=2-1, to=3-1]
\arrow[equals, from=2-2, to=3-2]
\arrow["{(\mathrm{id}, \eta\circ-)}", shift right, from=3-1, to=3-2]
\end{tikzcd}\]
commutes by functoriality of the tensor-Hom adjunction in the object. We show the top cell commutes by direct computation. Write $Y = X \otimes \mls O$.
Let $\varphi: f^*Y \to f^! \mls O$ be a morphism. Its image under the arrows going right and then down, or under the arrows going down and then right, is the composition
\[
Y \xrightarrow{\eta_Y} f_*f^*Y \xrightarrow{f_*\varphi}  f_*f^!\mls O \xrightarrow{\epsilon} \mls O.
\]

\section{Segre classes in exact sequences}\label{sec:intersection theory}

\begin{lemma}[cf. {\cite[Example 4.1.6]{fulton}}]
    \label{lemma: properties of segre}
    Let $X$ be an algebraic stack.
    For an exact sequence
                $0 \to E \xrightarrow{\psi} C \xrightarrow{\phi} C^{\prime} \to 0$
            of cones on $X$,
            we have
            \[
                s(C^{\prime}) = c(E) \cap s(C) \in A_{\ast}(X).
            \]
\end{lemma}

\begin{proof}
The proof of \cite[Example 4.1.6]{fulton} works. For completeness we replace \cite[Example 3.2.16(ii)]{fulton} with Lemma~\ref{lem: chern class regular section}, as this result is not immediate in Kresch's paper.
\end{proof}

\begin{lemma}
    \label{lem: chern class regular section}
    Let $Z = \VV(s) \subset X$ be the vanishing locus of a regular section
    $s$ of a rank $r$ vector bundle $E$ on a purely $d$-dimensional stack $X$.
    Then
    we have $p_{E}^{\ast}[Z] = s_{E\ast}[X]$ in the naive Chow group $A_{\ast}^{\circ}(E)$.
\end{lemma}

\begin{proof}
    Let $\tau$ be the tautological section of $p_{E}^{\ast}E$,
    and let $s^{\prime} = p_{E}^{\ast}s $.
    Let $\mc{Z} = \VV(\lambda \tau + \mu s^{\prime}) \subset E \times
    \PP^{1}_{[\lambda : \mu]}$.
    Then,
    the fibre of $\mc{Z}$ above $0 = [1:0]$ is the zero-section of $E$,
    whereas its fibre above $\infty = [0:1]$ is $p_{E}^{-1}(Z)$.
    Thus,
    $\mc{Z}$ would furnish a rational equivalence witnessing the desired
    equality once we are able to show that it is pure-dimensional,
    and has no components contained in any fibre.
    Since $\mc{Z} \setminus \mc{Z}_{\infty}$
    is isomorphic to $X \times \AA^{1}$ over $\AA^{1}$,
    it suffices to verify that $\mc{Z}^{\circ} = \mc{Z} \setminus \mc{Z}_{0}$
    has no components supported
    on the infinity fibre.

    The desired claim is local in nature,
    so we may reduce to the case where $X = \Spec R$
    is an affine scheme.
    Letting $t = \lambda/\mu$,
    we see that $\mc{Z}^{\circ}$ takes the form
    \[
        \mc{Z}^{\circ} = \Spec\frac{R[t, x_{1}, \dots, x_{r}]}{\left\langle tx_{1}-f_{1}, \dots, tx_{r}-f_{r}\right\rangle },
    \]
    with $f_{1}, \dots, f_{r} \in R$ being the regular sequence
    whose vanishing defines $Z$.
    Let $z \in \mc{Z}^{\circ}$ be an arbitrary point.
    A repeated application of Krull's Principal Ideal Theorem,
    combined with the pure-dimensionality of $X$,
    shows that $\dim \OO_{\mc{Z}^{\circ}, z} \geq
        \dim \OO_{E \times \AA^{1}, z} - r = (d + r + 1) - r = d + 1$ .
    In particular,
    any component of $\mc{Z}^{\circ}$ has dimension at least $d + 1$,
    and thus cannot be supported on $\mc{Z}_{\infty}$.

\end{proof}

\normalem
\printbibliography

\end{document}